\documentclass{article}

\usepackage{amsmath,amsthm, latexsym, amsfonts, amssymb, amscd,bbm}
\usepackage[english]{babel}
\usepackage[dvips]{graphicx}
\usepackage{mathrsfs}

\usepackage{color}

\newtheorem{theorem}{Theorem}[section]
\newtheorem{lemma}[theorem]{Lemma}

\newcommand{\norm}[1]{\left \| #1 \right \|}

 \newcommand{\lsup}[1]{\underset{#1\to\infty}{\overline{\lim}}}
\newcommand{\linf}[1]{\underset{#1\to\infty}{\underline{\lim}}}

\title{Large Deviations of Piecewise-Deterministic-Markov-Processes with Application to Stochastic Calcium Waves}
\author{Gaetan Barbet, James MacLaurin and Moshe Silverstein}

\begin{document}
\maketitle

\begin{abstract}
We prove a Large Deviation Principle for Piecewise Deterministic Markov Processes (PDMPs). This is an asymptotic estimate for the probability of a trajectory in the large size limit. Explicit Euler-Lagrange equations are determined for computing optimal first-hitting-time trajectories. The results are applied to a model of stochastic calcium dynamics. It is widely conjectured that the mechanism of calcium puff generation is a multiscale process: with microscopic stochastic fluctuations in the opening and closing of individual channels generating cell-wide waves via the diffusion of calcium and other signaling molecules. We model this system as a PDMP, with $N\gg 1$ stochastic calcium channels that are coupled via the ambient calcium concentration. We employ the Large Deviations theory to estimate the probability of cell-wide calcium waves being produced through microscopic stochasticity.
\end{abstract}

This paper determines a Large Deviation Principle for Piecewise-Deterministic Markov Processes (PDMPs), in the large population size limit. Piecewise-Deterministic Markov Processes (also known as stochastic hybrid systems) are used to model systems with multiple timescales \cite{Hasler2013,Azais2016}. They consist of a Markov Chain flipping between discrete states, coupled to an ordinary differential equation that evolves deterministically in-between the flips (and whose evolution typically depends on the state of the Markov Chain). The are also referred to as slow-fast systems. They enjoy numerous applications, particularly in biology \cite{Tyran2017,Bressloff2018c,Bressloff2021a,Bressloff2021}.  They have been used (for instance) to model excitable membranes in neuroscience \cite{Riedler2012, Lawley2016},  population dynamics in ecology \cite{Hening2023}, run-and-tumble dynamics of bacteria \cite{Berg1993} and stochastic models of calcium signals \cite{Keizer1998,Ramlow2023,Falcke2023b}.

Often, one wishes to estimate the probability of rare-events in these systems \cite{Kim2022}. For example, one may be interested in the probability of extinction of a species in population dynamics or a chemical reaction networks \cite{Bressloff2014a}, or the spontaneous and stochastic production of waves / pulses \cite{Keener2011}. Another well-known application - which is thoroughly explored in this paper - is  stochastic calcium dynamics (this will be discussed in more detail below). One of the best tools for estimating the probability of rare events is the theory of Large Deviations \cite{Budhiraja2019}. Large Deviations analysis has been performed on many stochastic differential equations and jump-Markov Processes, including \cite{He2014, Bouchet2016}. However it is comparatively less-developed for piecewise-deterministic Markov Processes.

To the author's knowledge, the first scholar to study the Large Deviations of PDMPs was Kifer \cite{Kifer2009}. This was followed by the work of Faggionato, Crivellari and Gabrielli \cite{Faggionato2009}, who determined a Large Deviation Principle for piecewise-deterministic Markov Processes. They require that (conditionally on the value of the `slow variable'), the Markov Process is ergodic. Kumar and Popovic also determined a Large Deviation Principle for stochastic hybrid systems \cite{Kumar2017a}. To do this, they used an ergodicity assumption to prove that the logarithmic moment generating function is sufficiently regular. Bressloff and Faugeras \cite{Bressloff2017a} demonstrated that the Large Deviations rate function can be written to have a Hamiltonian structure. \cite{Kraaij2020} proved a general Large Deviation Principle for slow-fast systems, by proving convergence of the Hamilton-Jacobi-Bellman equation.

On a related note, there has been much recent interest in the Large Deviations of Chemical Reaction Networks \cite{Patterson2019}. These are high-dimensional jump Markovian Processes, similar to the subject of this paper, except without the `slow' dynamics.  Recent works include those of Dupuis, Ramanan and Wu \cite{DupuisRamananWu2016}, Pardoux and Samegni-Kepgnou \cite{Pardoux2017},  Agazzi, Eckmann and Dembo \cite{Agazzi2018},  and Patterson and Renger \cite{Patterson2019}. Patterson and Renger \cite{Patterson2019} and Agazzi, Patterson, Renger and \cite{Agazzi2022} prove the Large Deviations Principle in a general setting by also studying the convergence of the reaction fluxes (like in this paper). Budhiraja and Friedlander \cite{BudhirajaFried2018} determine the Large Deviations of jump-Markov processes.

One of the novelties of the proof in this paper is its simplicity. We map the PDMP system to a set of homogeneous Poisson Processes via a time-rescaling. The Large Deviation Principle is then an almost immediate result of the Inverse Contraction Principle \cite{Dembo1998}. In addition, we explicitly determine the structure of the Euler-Lagrange equations (needed for computing optimal first-hitting times). This becomes an optimization problem over the time interval $[0,T]$, subject to a constraint (given by the slow processes). We then solve these equations to compute the probability of noise-induced calcium waves.

\subsection{Overview of Stochastic Models of Calcium Signaling}

Classical models of calcium signalling are almost entirely deterministic \cite{Keener2006}. These models typically assumed that the calcium concentration (and signalling molecules such as IP3) within the cell can be well-approximated as homogeneous, and therefore the dynamics can be accurately described by ordinary differential equations for the evolution of their concentrations in time \cite{Dupont2016}. However recent experimental evidence / mathematical analysis has called this into question \cite{Keizer1998a,Keener2006,Skupin2008,Keener2009b,Moenke2012,Rudiger2014,Ramlow2023,Falcke2023b}. It has been demonstrated that (i) the interspike interval can show significant variability, and (ii) intracellular calcium concentrations show steep gradients, and therefore it is widely postulated that much of the emergent phenomena are stochastic in nature \cite{Ramlow2023,Falcke2023b}. 

One of our fundamental aims is to develop an accurate microscopic model of calcium signalling (which is thoroughly stochastic) and then use statistical mechanical techniques to determine effective macroscopic equations. Calcium signaling is organized in a hierarchical manner \cite{Cao2013,Dupont2016,Friedhoff2021}. At the lowest (most microscopic) level, stochastic release of $Ca^{2+}$ upon binding of IP3 receptor results in a small localized increase in cytoplasmic $Ca^{2+}$ concentration, which is often termed a `blip'.  The IP3 receptors are typically tightly clustered, which means that a $Ca^{2+}$ blip can stimulate the release of additional $Ca^{2+}$ through neighboring receptors, so that the entire cluster emits a localized `puff'.  At the highest level of organization, if enough puffs are generated, they can form a propagating wave of increased $Ca^{2+}$ across an entire cell. There is considerable evidence that calcium puffs and waves are nonlinear stochastic phenomena: it has been observed for instance that IP3 channels can be highly active even when the average open probability is less than half its maximum \cite{Siekmann2012}. The very presence of abortive calcium waves is indicative of the stochastic origins of the phenomena. Indeed clusters have a diameter of $60-100$ nanometers whereas the distance between clusters is much larger: $3-7$ micrometers \cite{Dupont2016}. 

There is a well-established literature that examines how spatially-distributed calcium waves and patterns can arise out of microscopic stochastic models. Some of the earliest work is in the Fire-Diffuse-Fire model of Keizer and Smith \cite{Keizer1998,Keizer1998a}. In this model, Ryanodine Receptors open (and release calcium) once the ambient calcium exceeds a threshold; before closing and entering a refractory state before being ready to open again. Keizer and Smith were able to demonstrate that this model exhibits spatially-distributed waves.  Coombes, Hinch and Timofeeva \cite{Coombes2003,Coombes2004} developed a similar model. They estimated the release probability for a cluster of channels and determined that it is approximately a sigmoidal function of the local calcium concentration. Keener \cite{Keener2006} further extended these works by taking the opening and closing of individual channels to be stochastic. An advantage of all of the previous works is that the models are all spatially extended, unlike this paper. However the model in this paper fully couples the calcium concentration to the opening / closing of the channel. It thus allows a more detailed estimation of how the feedback between the opening / closing of channels and the local calcium concentration leads to calcium waves.



There have already been some efforts towards a detailed derivation of effective macroscopic equations from a microscopic model. 
 Hinch and Chapman \cite{Hinch2005} used exponential asymptotic methods were employed to determine the relative frequency of calcium sparks. Falcke \textit{et al} employed approximations to determine estimates for the approximate probability of calcium puffs (given the frequency of blips), and also estimates for the probability of a wave throughout the cell \cite{Moenke2012,Dupont2016}.

This paper is structured as follows. In Section I, we introduce the model and outline our main results. In Section II, we determine the Euler-Lagrange equations that follow from the Large Deviation Principle. These can be used to estimate optimal trajectories for first-hitting-time problems. In Section III we apply the theory to estimate the probability of the spark-to-wave transition in calcium dynamics. Finally in Section IV, we outline the proofs.

\textit{Notation:}\\
Let $\norm{\cdot}_t$ denote the supremum norm over $[0,t]$, i.e.
\[
\norm{z}_t = \sup_{s\in [0,t]}|z_s|.
\]
Let $D([0,t],\mathbb{R})$ be the cadlag space of real functions from $[0,t]$ to $\mathbb{R}$ that are right-continuous and have left-hand limits.  Define $\Lambda_t$ to be the class of strictly increasing, continuous mappings of $[0,t]$ onto itself (so that any $\lambda\in\Lambda_t$ is such that $\lambda(0) = 0$ and $\lambda(t) = t$. Following Billingsley \cite{Billingsley1999}, for any $\lambda \in \Lambda$, define
\begin{equation}\label{eq: Lipschitz rescaling Billingsley}
\norm{\lambda}^{\circ}_t =\sup_{s< t}\bigg|\log \frac{\lambda(t) - \lambda(s)}{t-s} \bigg|
\end{equation}
Define the Skorohod Metric on $D_t$,
\begin{equation}
d^{\circ}_t(x,y) = \inf_{\lambda \in \Lambda_t} \big\lbrace \norm{\lambda}^{\circ}_t \wedge \sup_{s\in [0,t]}| x(s) - y(\lambda_j(s)) | \big\rbrace .
\end{equation}

It is proved in Billingsley \cite{Billingsley1999} that $d^{\circ}_t$ is a metric, and that $D_t$ is complete and separable with respect to the topology induced by $d^{\circ}_t$. Next, for $a\in \mathbb{Z}^+$ define the pseudo-metric $\tilde{d}^{\circ}_a: \mathcal{D}([0,\infty),\mathbb{R})\times \mathcal{D}([0,\infty),\mathbb{R}) \to \mathbb{R}^+$ to be 
\begin{align}
\tilde{d}^{\circ}_a(x,y) =& d^{\circ}_a( g_a x, g_a y) \text{ where }\\
g_a(t) =& \left\lbrace \begin{array}{c}
1 \text{ if }t\leq a-1 \\
0 \text{ if }t\geq a \\
a-t \text{ if }a-1 \leq t \leq a 
\end{array} \right. \label{eq: g a definition}
\end{align}
Then define the metric $d^{\circ}_{\infty}: \mathcal{D}([0,\infty),\mathbb{R})\times \mathcal{D}([0,\infty),\mathbb{R}) \to \mathbb{R}^+$ to be
\begin{equation}
d^{\circ}_{\infty}(x,y) = \sum_{a=1}^{\infty} \inf\big\lbrace 2^{-a} , \tilde{d}^{\circ}_a(x,y) \big\rbrace. \label{eq: d circ infinity metric}
\end{equation}
It is proved in Billingsley \cite{Billingsley1999} that $d^{\circ}_{\infty}$ metrizes convergence over $\mathcal{D}([0,\infty),\mathbb{R})$.

\section{Model Outline and Large Deviations Principle}


A chemical reaction network \cite{Anderson2015} consists of the following triple:
\begin{itemize}
\item \textit{species} $\mathscr{S}$, which are the chemical components whose counts we wish to model dynamically. There are $d$ of these.
\item \textit{complexes} $\mathscr{C}$ which are nonnegative linear combinations of the species that describe how the species can interact and
\item \textit{reactions} $\mathscr{R}$ which describe how to convert one complex to another.
\end{itemize}
Let $X(t) \in \mathbb{Z}_+^d$ be a vector that counts the number of particles of each species. For a reaction $\alpha\in\mathscr{R}$, let $Z_{\alpha}(t) \in \mathbb{Z}_+$ count the number of times the reaction has occured. We thus have that
\begin{equation}
X(t)  =X(0) + \sum_{\alpha \in \mathscr{R}}Z_{\alpha}(t) \xi_{\alpha} 
\end{equation}
where $\xi_\alpha \in \mathbb{Z}^d$ is the reaction vector corresponding to the $\alpha^{th}$ reaction. We define the concentration vector to be the scaled number of particles, i.e.
\begin{align}
x(t) &=  N^{-1} X(t) = N^{-1}X(0)+  \sum_{\alpha \in \mathscr{R}}z_{\alpha}(t) \xi_{\alpha} , \label{eq: x definition}
\end{align}
where $N \gg 1$ is a scaling parameter that indicates the system size (for instance, the total number of particles at time $0$).

The intensity of the $\alpha^{th}$ reaction is defined to be
\begin{equation}
N\lambda_\alpha\big( x(t) ,  u(t) \big), \label{eq: intensity of N system}
\end{equation}
where $\lambda_\alpha: \mathbb{R}^d \times \mathbb{R}^m \to \mathbb{R}^+$ is continuous and bounded. It is assumed that reactions cannot produce a negative concentration: that is, it is assumed that if $x\in \mathbb{R}_+^d$ and $u \in \mathbb{R}^m$, and $\alpha\in \mathscr{R}$ are such that $x + N^{-1}\xi_{\alpha} \notin \mathbb{R}_+^d$, then necessarily $\lambda_{\alpha}(x,u) =0$.


The ordinary differential equation takes values in a state space $\mathbb{R}^m$, and is such that
\begin{align}\label{eq: u ODE}
\frac{du}{dt} = A\big( u(t), x(t) \big),
\end{align}
where $A: \mathbb{R}^m \times \mathbb{R}^d \to \mathbb{R}^d$ is continuous. The initial conditions are taken to be non-random constants i.e. 
\begin{align}
u(0) &:=  \hat{u}^N_0 \text{ and }
x(0) :=  \hat{x}^N_0.
\end{align}
and it is assumed that the following limits exist
\begin{align}
\lim_{N\to\infty} \hat{u}^N_0 &= \hat{u}_0 \\
\lim_{N\to\infty} \hat{x}^N_0 &= \hat{x}_0.
\end{align}
We assume the following uniform growth bounds: for some constant $K > 0$, for all $t\geq 0$, and $u\in \mathbb{R}^m$, and all $\alpha \in \mathscr{R}$,
\begin{align}
\norm{ A( u , X  ) } &\leq K \big(1 + \norm{u} + \norm{X}\big) \label{eq: bound the u intensity} \\
\big| \lambda_\alpha\big( X , u \big) \big| &\leq K.
\end{align}
Second, we assume that $A$ and $\lambda_n$ are both uniformly Lipschitz in their arguments: that is,
\begin{align}
\norm{ A\big( u, X \big) -A\big( \tilde{u}, \tilde{X} \big)   } &\leq C \big\lbrace \norm{u - \tilde{u}} + \norm{X - \tilde{X}} \big\rbrace \\
\sup_{\alpha \in \mathscr{R}}\| \lambda_\alpha\big( u, X \big) -\lambda_\alpha\big( \tilde{u}, \tilde{X} \big)   \| &\leq C\big\lbrace \norm{u - \tilde{u}} + \norm{X - \tilde{X}} \big\rbrace .
\end{align}
Standard ODE theory dictates that for any piecewise-continuous $x$, there is a unique solution to the ODE in \eqref{eq: u ODE}.

 The linearity of Poisson Processes ensures that $Z_{\alpha}(t)$ can be represented as, writing $\lbrace Y_{\alpha}(t) \rbrace_{\alpha \in \mathscr{R}}$ to be independent unit-intensity Poisson Processes \cite{Anderson2015},
\begin{align}
Z_{\alpha}(t) 
=&  Y_{\alpha}\bigg(N \int_0^t  \lambda_{\alpha}\big( x(s) , u(s) \big) ds \bigg)\label{eq: Z definition} \\
z_{\alpha}(t) =& N^{-1}Z_{\alpha}(t) \text{ with initial condition} \\
z_{\alpha}(0) =& 0.
\end{align}
Writing $y_{\alpha}(t) = N^{-1}Y_{\alpha}(Nt)$, it may be observed that
\begin{equation}
z_{\alpha}(t) = y_{\alpha}\bigg(  \int_0^t  \lambda_{\alpha}\big( x(s) , u(s) \big) ds \bigg).\label{eq: Z definition next} 
\end{equation}
 
The main result of this paper is the following theorem. Further below in \eqref{Upsilon Definition} we specify $\Upsilon$ to be the state space that the triple $(z,x,u)$ lives in (and we define a topology too). The sequence of probability laws of $\big( z,x,u \big)$ satisfies the following asymptotic estimate on the space $\Upsilon$ as $N\to\infty$.
\begin{theorem}\label{Theorem Main LDP}
Supponse that $\mathcal{O},\mathcal{A} \subset \Upsilon$, with $\mathcal{O}$ open and $\mathcal{A}$ closed, are such that for any $T > 0$,
\begin{align}
\inf_{ (z,x,u) \in \mathcal{A} \cup \mathcal{O}} \inf_{\alpha\in\mathscr{R}} \inf_{s\in [0,T]} \lambda_\alpha(x(s),u(s)) > 0. \label{eq: set infimum lambda first statement}
\end{align}
Then there exists a function $\mathcal{J}: \Upsilon \to \mathbb{R}$ (specified in Section \ref{Section Topological Definitions}) that is (i) lower-semi-continuous and (ii) has compact level sets, such that
\begin{align}
\lsup{N}N^{-1}\log \mathbb{P}\big( (z,x,u) \in \mathcal{A} \big) &\leq - \inf_{(z,x,u) \in \mathcal{A}}\mathcal{J}(z,x,u) \\
\linf{N}N^{-1}\log \mathbb{P}\big( (z,x,u) \in \mathcal{O} \big) &\geq - \inf_{(z,x,u)\in \mathcal{O}}\mathcal{J}(z,x,u).
\end{align}
 \end{theorem}

\subsection{Topological Definitions} \label{Section Topological Definitions}
Before we state our main result we must briefly note the topological space that our variables belong to. Write the state space for $\big( z,x, u \big)$ as $\Upsilon$, i.e.
\begin{multline}\label{Upsilon Definition}
\Upsilon :=  \big\lbrace (z,x,u) \in \tilde{\mathcal{D}}([0,\infty),\mathbb{R}^+)^{M} \times \mathcal{D}([0,\infty),\mathbb{R}^+)^{d} \times \mathcal{C}^1([0,\infty),\mathbb{R})^m \\  \text{ where (i)}\sum_{\alpha\in\mathscr{R}}  z_{\alpha}(t) \zeta_{\alpha} = x(t) ,\; \text{(ii) for all }t\geq 0, \; \; \frac{du}{dt} = A\big( u(t), x(t) \big) \text{ and }u(0) = \hat{u} 
 \big\rbrace
\end{multline}
and endow $\Upsilon$ with the product topology. Since the values of $x$ and $u$ are effectively determined by $z$, we only need to metrize the convergence of $z$. This is captured by the following lemma.

\begin{lemma}\label{Lemma Topology on Upsilon}
$\Upsilon$ is a closed subset of  $\tilde{\mathcal{D}}([0,\infty),\mathbb{R}^+)^{M} \times \mathcal{D}([0,\infty),\mathbb{R}^+)^{d} \times \mathcal{C}^1([0,\infty),\mathbb{R})^m$. Also, $\Upsilon$ is separable and metrizable. Furthermore for any $t > 0$, there exists $c_t > 0$ such that for all $\delta > 0$,
\begin{multline}
\big\lbrace (z,x,u), (\tilde{z},\tilde{x},\tilde{u}) \in \Upsilon \times \Upsilon : \sup_{\alpha\in\mathscr{R}}d^{\circ}_t(z_{\alpha},\tilde{z}_{\alpha}) \leq \delta \big\rbrace \subseteq  
\big\lbrace (z,x,u), (\tilde{z},\tilde{x},\tilde{u}) \in \Upsilon \times \Upsilon : d^{\circ}_t(z,\tilde{z}) \leq \delta \text{ and }\\ \sup_{0\leq s \leq t}\norm{x(s) - \tilde{x}(s)} \leq c_t \delta \text{ and }\sup_{0\leq s \leq t}\norm{u(s) - \tilde{u}(s)} \leq c_t \delta \big\rbrace 
\end{multline}
\end{lemma}
We can now define the rate function
\begin{align}
\mathcal{J}(z,x,u): \Upsilon \mapsto \mathbb{R}^+,
\end{align}
starting by stipulating that it is infinite in the case that $z_{\alpha}: \mathbb{R}^+ \to \mathbb{R}^+$ is not absolutely continuous for any $\alpha\in\mathscr{R}$. 
Otherwise, if $z_{\alpha}$ is absolutely continuous for every $\alpha \in \mathscr{R}$  (which means that it must have a derivative $\dot{z}_{\alpha}$ for Lebesgue almost every time), for any realization of $z$, define the set
\begin{equation} \label{eq: set of Lebesgue measure zero}
\mathfrak{V}(z,x,u) = \big\lbrace t\in \mathbb{R}^+: \text{For any }\alpha\in\mathscr{R}, \; \dot{z}_{\alpha}(t) \neq 0 \text{ and }\lambda_{\alpha}( x(t), u(t) ) = 0 \big\rbrace
\end{equation}
and define the rate function
\begin{align}
\mathcal{J}\big( z,x,u \big) &=  \infty \text{ if }\mathfrak{V}(z,x,u)  \text{ is of nonzero Lebesgue Measure, else otherwise }\\
&= \sum_{\alpha\in \mathscr{R}}\int_{\mathbb{R}^+ / \mathfrak{V}(z,x,u) } \ell\bigg( \dot{z}_{\alpha}(r) / \lambda_{\alpha}(x(r),u(r)) \bigg)\lambda_{\alpha}(x(r),u(r))  dr\text{ where }  \label{eq: J rate function}\\
 \ell(a) &= a\log a -a +1.\label{eq: ell definition}
  \end{align}

\section{Euler-Lagrange Equations}

In this section we determine the structure of the Euler-Lagrange equations that determine the most likely trajectory followed by the system in attaining a specified state. Computing estimates for first-hitting-times is one of the most important applications of Large Deviations Theory. For the application of Large Deviations Theory to first-hitting-time estimates, see for example \cite{E2010,Grafke2017,Newby2024}.   The `fast processes' (i.e. the ODE in $u(t)$) are implemented as a constraint (with Lagrange multipliers). We start by framing the problem in terms of the reaction fluxes; that is, we we wish to determine the most likely trajectory followed by the system in attaining a target reaction flux  $z^* \in (\mathbb{R}^+)^{|\mathscr{R}|}$ at time $T$. We must determine the trajectory 
\[
(z,x,u) \in  \mathcal{C}^1(\mathbb{R})^M \times   \mathcal{C}^1(\mathbb{R})^d \times   \mathcal{C}^1(\mathbb{R})^m
\]
such that $\mathcal{J}_T$ is minimized, where
\begin{align}\label{eq: restated rate function}
\mathcal{J}_T = \sum_{\alpha\in\mathscr{R}} \int_0^T L\big( \dot{z}(s), x(s) , u(s) \big) ds,
\end{align}
where $L: \mathbb{R}^M \times \mathbb{R}^d\times  \mathbb{R}^m  \ \to \mathbb{R}^+$ is the integrand of the rate function in \eqref{eq: J rate function}, i.e.
  \begin{align}
  L\big(q,x , u \big)&=  \sum_{\alpha\in \mathscr{R}} \ell\big( q_{\alpha} / \lambda_{\alpha}(x,u) \big)\lambda_{\alpha}(x,u)  \\
  \ell(a) &= a\log a -a +1.\label{eq: ell definition restated}
  \end{align}
subject to the following constraints and boundary conditions
\begin{align}
\frac{du}{dt} =& A(u,x) \\
x(t) =&  x(0) +  \sum_{\alpha \in \mathscr{R}}z_{\alpha}(t) \xi_{\alpha} . \label{eq: x definition restated}\\
u(0) =& \hat{u}\\
z_{\alpha}(0) =& 0 \\
z_{\alpha}(T) =& z_{\alpha}^*.
\end{align}
We introduce Lagrange Multipliers $\lbrace \eta^i(t) \rbrace_{1\leq i \leq m}$ corresponding to the $m$ constraints
  \begin{align}
  \frac{du^i_t}{dt} = A^i(u_t, x_t).
  \end{align}
  Our task is now to find the critical points of the functional
  \begin{align}
\tilde{\mathcal{J}}_T( z,u,\eta) =& \int_0^T \tilde{L}\big( \dot{z}(s), x(s) , u(s) \big) ds \text{ where } \label{eq: Functional with constraint}\\
\tilde{L}\big( \dot{z}(s), x(s) , u(s) \big) =&  L\big( \dot{z}(s), x(s) , u(s) \big) -\sum_{i=1}^m  \eta^i(s) \big\lbrace \dot{u}^i(s)- A^i(u(s),x(s)) \big\rbrace  \text{ such that }\\
x(t) =&  \hat{x}(0) +  \sum_{\alpha \in \mathscr{R}}z_{\alpha}(t) \xi_{\alpha} . \label{eq: x definition restated 2}
  \end{align}
 We now take the Frechet Derivative of \eqref{eq: Functional with constraint}, i.e. for $w \in \mathcal{C}^1([0,T],\mathbb{R})^M$ and $v \in \mathcal{C}^1([0,T],\mathbb{R})^m$, and $\kappa \in \mathcal{C}^1([0,T],\mathbb{R})^m$, with boundary conditions $w(0) = 0, w(T) = 0, v(0) = 0$, and the boundary conditions for $\kappa$ are not yet clear. Thanks to the constraint in \eqref{eq: x definition restated 2}, we have that
 \begin{align}
 \frac{d}{dz_{\alpha}} = \sum_{i=1}^d \frac{\partial }{\partial x^i} \xi^i_{\alpha}.
 \end{align}
 We therefore find that
 \begin{multline}
D \tilde{\mathcal{J}}_T( z,u,\eta) \cdot (w,v,\kappa) :=  \lim_{\epsilon \to 0^+}\epsilon^{-1} \big\lbrace \tilde{\mathcal{J}}_T( z + \epsilon w,u + \epsilon v,\eta + \epsilon \kappa) -  \tilde{\mathcal{J}}_T( z  ,u  ,\eta  ) \big\rbrace \\
=\sum_{\alpha\in\mathscr{R}} \int_0^T w_{\alpha}(t) \bigg\lbrace \sum_{i=1}^d \bigg(  \frac{\partial L}{\partial x^i}\xi^i_{\alpha}  +  \sum_{k=1}^m \eta^k \frac{\partial A^k}{\partial x^i} \xi^i_{\alpha} \bigg)  - \frac{d}{dt} \frac{\partial L}{\partial \dot{z}_{\alpha}} \bigg\rbrace dt \nonumber \\
+ \sum_{i=1}^m\int_0^T v^i(t) \bigg\lbrace \frac{\partial L}{\partial u^i} + \frac{d\eta^i}{dt} + \sum_{k=1}^m \eta^k \frac{\partial A^k}{\partial u^i}  \bigg\rbrace  
-\kappa^i \bigg\lbrace \frac{du^i}{dt} - A^i(u(t),x(t)) \bigg\rbrace dt  \\
+ \sum_{\alpha\in\mathscr{R}}\bigg\lbrace w_{\alpha}(T)\frac{\partial }{\partial\dot{z}_{\alpha}}L(\dot{z}(T),x(T), u(T)) - w_{\alpha}(0)\frac{\partial }{\partial\dot{z}_{\alpha}}L(\dot{z}(0),x(0), u(0))\bigg\rbrace \\
-  \sum_{i=1}^m \big\lbrace \eta^i(T) v^i(T) - \eta^i(0) v^i(0) \big\rbrace , 
 \end{multline}
 and to obtain the above expression, we have used integration by parts to write
 \begin{multline}
 \sum_{\alpha\in\mathscr{R}} \int_0^T \dot{w}_{\alpha}(t) \frac{\partial L}{\partial\dot{z}_{\alpha}} dt =  \sum_{\alpha\in\mathscr{R}}\bigg\lbrace w_{\alpha}(T)\frac{\partial }{\partial\dot{z}_{\alpha}}L(\dot{z}(T),x(T), u(T)) - w_{\alpha}(0)\frac{\partial }{\partial\dot{z}_{\alpha}}L(\dot{z}(0),x(0), u(0)) \\ -  \int_0^T  w_{\alpha}(t)\frac{d}{dt} \frac{\partial L}{\partial \dot{z}_{\alpha}} dt  \bigg\rbrace
 \end{multline}
 and
 \begin{equation}
   \sum_{i=1}^m \int_0^t \eta^i(t) \dot{v}^i(t) dt =   \sum_{i=1}^m \big\lbrace \eta^i(T) v^i(T) - \eta^i(0) v^i(0)  -\int_0^t \dot{\eta}^i(t) v^i(t) dt  \big\rbrace .
 \end{equation}
 Setting the coefficients of $w_{\alpha}$, $v$ and $\kappa^i$ to zero, we obtain the equations
 \begin{align}
\sum_{i=1}^d  \bigg(\frac{\partial L}{\partial x^i}\xi^i_{\alpha}  +  \sum_{k=1}^m \eta^k \frac{\partial A^k}{\partial x^i} \xi^i_{\alpha} \bigg)  - \frac{d}{dt} \frac{\partial L}{\partial \dot{z}_{\alpha}} &= 0 \\
  \frac{\partial L}{\partial u^i}+ \frac{d\eta^i}{dt}  + \sum_{k=1}^m \eta^k \frac{\partial A^k}{\partial u^i} &= 0 \\
   \frac{du^i}{dt} - A^i(u(s),x(s) )&= 0 \\
  \eta^i(T) &= 0.
 \end{align}
 Noting that $\dot{\ell}(a) = \log a$, we compute that $\frac{\partial L}{\partial \dot{z}_{\alpha}} = \log \frac{\dot{z}_{\alpha}}{\lambda_{\alpha}(x(t),u(t))}$, and therefore
 \begin{align}
 \frac{d}{dt} \frac{\partial L}{\partial \dot{z}_{\alpha}} =& \frac{d}{dt}\left( \log \frac{\dot{z}_{\alpha}}{\lambda_{\alpha}(x(t),u(t))}  \right) \\
 =& \frac{\ddot{z}_{\alpha}}{\dot{z}_{\alpha}} - \frac{1}{\lambda_{\alpha}(x(t),u(t))} \frac{d}{dt}\lambda_{\alpha}(x(t),u(t)).
  \end{align}
  Thus in the case that $\lambda_{\alpha}(x,u) > 0$, the Euler Lagrange equations are such that for $1\leq j \leq m$, $\alpha \in \mathscr{R}$ and $1\leq i \leq d$,
    \begin{align}
 \frac{\ddot{z}_{\alpha}}{\dot{z}_{\alpha}} =& \frac{1}{\lambda_{\alpha}} \frac{d\lambda_{\alpha}}{dt}+\sum_{i=1}^d \xi^i_{\alpha} \frac{\partial L}{\partial x_i} - \sum_{k=1}^d \sum_{j=1}^m \eta^j(t) \frac{\partial A^j}{\partial x^k_t} \xi^k_{\alpha} \label{eq: EL summary 1}\\
  \frac{d\eta^j}{dt} =& - \frac{\partial L}{\partial u^j} - \sum_{k=1}^m \eta^k \frac{\partial A^k}{\partial u^j} \\
  x^i(t) =&  x^i(0) +  \sum_{\alpha \in \mathscr{R}}z_{\alpha}(t) \xi^i_{\alpha} \text{ for all }t\geq 0 . \label{eq: x definition restated again}\\
  \eta^j (T) =& 0 \\
  z_{\alpha}(0) =& 0 \\
  z(T) =& z_*. \label{eq: EL summary 2}
    \end{align}
    For the sake of completeness, we note that
    \begin{align}
    \frac{d\lambda_{\alpha}}{dt} =& \sum_{\beta \in \mathscr{R}} \sum_{i=1}^d \frac{\partial\lambda_{\alpha}}{\partial x^i}\xi^i_{\beta} \dot{z}_{\beta}+\sum_{i=1}^m \frac{\partial\lambda_{\alpha}}{\partial u^i}A^i(u(t),x(t)) \\
     \frac{\partial L}{\partial u^i} =& \sum_{\alpha\in\mathscr{R}}\frac{\partial \lambda_{\alpha}}{\partial u^i} \bigg(1 - \frac{\dot{z}_{\alpha}}{\lambda_{\alpha}} \bigg)
    \end{align}
 It can be seen that this system amounts to an ODE boundary value problem. One way that one may attempt to solve this problem is via the shooting method \cite{E2010}. At time $0$, the unknowns are $\lbrace \dot{z}_{\alpha}(0) \rbrace_{\alpha \in \mathscr{R}}$ and $\lbrace \eta^j(0) \rbrace_{1\leq j \leq m}$. The shooting method requires one to guess values for the unknowns at time $0$, (ii) integrate forward the equations to time $T$, and then (iii) implement the constraints $\lbrace \dot{z}_{\alpha}(T) = \dot{z}_{*,\alpha} \rbrace_{\alpha \in \mathscr{R}}$  and $\lbrace \eta^j(T) = 0 \rbrace_{1\leq j \leq m}$.
   

 \subsection{Simplified Euler-Lagrange Equations}
Usually, one desires to know the first-hitting-time for the concentrations $\lbrace x^i(T) \rbrace_{1\leq i \leq d}$, rather than the first hitting time for the reaction fluxes. That is, for a fixed $x_* \in \mathbb{R}^d$, one wishes to estimate the probability that $x_T \simeq x_*$, and also determine the optimal (most likely) path followed by the system in attaining this point. For this reason, it can be simpler to eliminate the reaction fluxes and reduce the problem to one purely in terms of the concentrations and the ODE variables $\lbrace u^i(T) \rbrace_{1\leq i \leq m}$. To this end, for any $\dot{x} \in \mathbb{R}^d$, $x \in (\mathbb{R}^+)^d$, and $u \in \mathbb{R}^m$, define the set
 \begin{align}
 \Xi(\dot{x},x,u) = \big\lbrace \dot{z} \; : \; \dot{x}^p =\sum_{\alpha\in\mathscr{R}} \xi^p_{\alpha} \dot{z}_{\alpha} \big\rbrace  \subset (\mathbb{R}^+)^M
 \end{align}
 Then define the function
 \begin{equation}
\hat{L}\big( \dot{x}(s), x(s) , u(s) \big) = \inf \big\lbrace L( \dot{z}, x(s) , u(s) ) \; : \; \dot{z}  \in  \Xi(\dot{x},x,u) \big\rbrace \label{eq: contracted rate function}
 \end{equation}
 In the case that $ \Xi(\dot{x},x,u)  = \emptyset$, define $\hat{L}\big( \dot{x}(s), x(s) , u(s) \big) = \infty$. One can thus define the contracted rate function
\begin{align}
\tilde{\mathcal{J}}_T:& \mathcal{C}_{ac}([0,T],(\mathbb{R}^+)^d \big) \times \mathcal{C}_{ac}([0,T],\mathbb{R}^m \big) \to \mathbb{R} \\
\tilde{\mathcal{J}}_T(x,u) &=  \int_0^T \tilde{L}\big( \dot{x}(s), x(s) , u(s) \big) ds,
\end{align}
and it is immediate from the Contraction Principle \cite{Dembo1998} that $\tilde{\mathcal{J}}_T$ governs the Large Deviations of $(x,u)$. The following convexity property makes it easier to compute optimal trajectories.
 \begin{lemma}
 $\tilde{L}$ is convex in its first argument.
 \end{lemma}
 \begin{proof}
This is because $\Xi$ is linear in $\dot{x}$,  and $L$ is convex in its first argument.
 \end{proof}
Next, we outline the Euler-Lagrange equations for $\tilde{\mathcal{J}}_T$. The derivation parallels that of the previous section. Our task is to find the critical points of the functional
  \begin{align}
\grave{\mathcal{J}}_T( x,u,\eta) =& \int_0^T \grave{L}\big( \dot{x}(s), x(s) , u(s) \big) ds \text{ where } \label{eq: Functional with constraint 2}\\
\grave{L}\big( \dot{x}(s), x(s) , u(s) , \eta(s) \big) =&  \tilde{L}\big( \dot{x}(s), x(s) , u(s) \big) -\sum_{i=1}^m  \eta^i(s) \big\lbrace \dot{u}^i(s)- A^i(u(s),x(s)) \big\rbrace ,
  \end{align}
  and $\eta^i \in \mathcal{C}^1([0,T],\mathbb{R})$ is the Lagrange Multiplier.
  
 We now take the Frechet Derivative of \eqref{eq: Functional with constraint 2}, i.e. for $w \in \mathcal{C}^1([0,T],\mathbb{R})^M$ and $v \in \mathcal{C}^1([0,T],\mathbb{R})^m$, and $\kappa \in \mathcal{C}^1([0,T],\mathbb{R})^m$, with boundary conditions $w(0) = 0, w(T) = 0, v(0) = 0$, and the boundary conditions for $\kappa$ are not yet clear,
 \begin{multline}
D \grave{\mathcal{J}}_T( x,u,\eta) \cdot (w,v,\kappa) :=  \lim_{\epsilon \to 0^+}\epsilon^{-1} \big\lbrace \grave{\mathcal{J}}_T( x + \epsilon w,u + \epsilon v,\eta + \epsilon \kappa) -  \tilde{\mathcal{J}}_T( x  ,u  ,\eta  ) \big\rbrace \\
= \int_0^T \sum_{i=1}^d  w_{i}(t) \bigg\lbrace \bigg(  \frac{\partial \hat{L}}{\partial x^i} +\sum_{k=1}^m  \eta^k  \frac{\partial A^k}{\partial x^i} \bigg)  - \frac{d}{dt} \frac{\partial L}{\partial \dot{x}^i} \bigg\rbrace dt \nonumber \\
+ \sum_{i=1}^m\int_0^T v^i(t) \bigg\lbrace \frac{\partial \hat{L}}{\partial u^i} + \frac{d\eta^i}{dt} + \sum_{k=1}^m \eta^k \frac{\partial A^k}{\partial u^i}  \bigg\rbrace  
-\kappa^i \bigg\lbrace \frac{du^i}{dt} - A^i(u(t),x(t)) \bigg\rbrace dt  \\
-  \sum_{i=1}^m \big\lbrace \eta^i(T) v^i(T) - \eta^i(0) v^i(0) \big\rbrace
 \end{multline}
The Euler-Lagrange equations are thus, for $1\leq i \leq d$ and $1\leq j \leq m$,
 \begin{align}
\frac{\partial \hat{L}}{\partial x^i} + \sum_{k=1}^m \eta^k  \frac{\partial A^k}{\partial x^i}  - \frac{d}{dt} \frac{\partial \hat{L}}{\partial \dot{x}^i} &= 0 \label{eq: reduced EL 1} \\
  \frac{\partial \hat{L}}{\partial u^j}+ \frac{d\eta^j}{dt}  + \sum_{k=1}^m \eta^k \frac{\partial A^k}{\partial u^j} &= 0 \\
   \frac{du^i}{dt} - A^i(u(s),x(s) )&= 0 \\
  \eta^i(T) &= 0 .\label{eq: reduced EL 4}
 \end{align}
For the sake of completeness, we note that the derivative assumes the form
\begin{align}
\frac{d}{dt} \frac{\partial \hat{L}}{\partial \dot{x}^i} =\sum_{j=1}^d \bigg\lbrace \frac{\partial^2 \tilde{L}}{\partial \dot{x}^i \partial \dot{x}^j} \ddot{x}^j +  \frac{\partial^2 \hat{L}}{\partial \dot{x}^i \partial x^j} \dot{x}^j \bigg\rbrace+\sum_{j=1}^m  \frac{\partial^2 \hat{L}}{\partial \dot{x}^i \partial u^j}  \frac{du^j}{dt}
\end{align}

\section{Application to Calcium Signalling}

We apply this theory to determine estimates for the typical time it takes for cells with stochastically opening and closing calcium channels to change from most of the channels being closed to most being open \cite{Dupont2016}. There are thus two components to the model: the opening and closing channels, and the calcium concentration (that evolves deterministically between the opening and closing of the channels). The model is simplified to be such that the calcium concentration throughout the cell is spatially homogeneous. 

To begin with, we must outline a simple microscopic model of stochastic calcium effects. Following for instance \cite{Coombes2003,Keener2006,Ramlow2023,Falcke2023b}, we employ a hybrid model: with the calcium diffusion modeled deterministically, the IP3 concentration constant throughout the cell, and the channel opening and closing modeled stochastically. 


Our basic aim is to determine the probability of calcium puffs and waves by studying a similar stochastic model to  \cite{Coombes2003,Keener2006,Ramlow2023,Falcke2023b}. These papers performed detailed numerical simulations, and analytically estimated the probabilities. To calculate the probability of a calcium puff in a cluster, the calcium and IP3 concentration throughout the cluster is assumed to be homogeneous. However there is a feedback effect of the opening / closing of the channels: when these bind or unbind the calcium / IP3, they alter the local concentration. Thus the opening and closing of channels within a cluster is not independent because the channels communicate with each other via the calcium / IP3 concentrations.

\subsection{Piecewise-Deterministic-Markov-Process Model of Calcium Dynamics}

We suppose that there are $N$ channels distributed on the cell membrane. It is assumed that spatial effects are negligible, so that the calcium concentrations can be modelled as approximately spatially homogeneous using ODEs. Let $u_1$ be the calcium concentration in the cellular cytosol, and $u_2$ be the calcium concentration in the endoplasmic reticulum.

We take the number of stochastic variables corresponding to each channel to be $1$. That is, $Z^i_t = 1$ if the channel is open, and $Z^i_t = 0$ if the channel is closed.  Employing a standard model \cite{Dupont2016} for the calcium dynamics, we obtain that
\begin{align}
\frac{du_1}{dt} =& k_f x(t) (u_2 - u_1) - J_{serca}(u) +  J_{leak} := A^1(u,x) \label{eq: ODE 1 restated}\\
\frac{du_2}{dt} =& \gamma J_{serca}(u_1) - \gamma  k_f x(t) (u_2 - u_1) - \gamma J_{leak} := A^2(u,x), \label{eq: ODE 1 restated} \\
J_{serca}(u_1) =& \frac{V_s (u_1)^2}{K_s^2 + (u_1)^2} \\
J_{leak}(u) =& k_{leak}(u_2 - u_1)\label{eq: ODE 4}
\end{align}
One may observe that 
\begin{align}
\frac{d}{dt}\big( \gamma u_1 + u_2 \big) = 0,
\end{align}
which means that we can eliminate $u_2$ from the dynamics, and write $u := u_1(t)$, and $A(t) := A^1(u,x)$.
Now 
\begin{align}
x(t) = \frac{1}{N}\sum_{j=1}^N Z^j_t.
\end{align}
The Markovian switching of the channel is given by
\begin{itemize}
\item Reaction 1: $1$ to $0$ with intensity $\lambda_1(u,x) = \alpha_{-1}x$ (a constant). The stoichiometric constant is $\zeta_1 = -1$.
\item Reaction 2: $0$ to $1$ with intensity $\lambda_2(u,x) = \alpha_1 u_1 (1-x)$. The stoichiometric constant is $\zeta_2 = 1$.
\end{itemize}
Let $z_1(t)$ count the number of times reaction $1$ happens, and $z_2(t)$ count the number of times that reaction $2$ happens. Then it must be that
\begin{align}
x(t) = x(0) + z_2(t) - z_1(t).
\end{align}

\subsubsection{Large $N$ Limiting Dynamics.}

As $N\to\infty$ it is classical that the concentration of open channels converges to the following ODE \cite{Kurtz1971}.
\begin{align}
\frac{dz}{dt} =& \lambda_2(u,x) - \lambda_1(u,x). \label{eq: large N channel proportion} \\
\frac{du_1}{dt} =& k_f x(t) (u_2 - u_1) - J_{serca}(u) +  J_{leak} := A^1(u,x) \label{eq: ODE 1 restated again}\\
\frac{du_2}{dt} =& \gamma J_{serca}(u_1) - \gamma  k_f x(t) (u_2 - u_1) - \gamma J_{leak} := A^2(u,x), \label{eq: ODE 1 restated again} 
\end{align}
It turns out that there exists a unique fixed point of \eqref{eq: large N channel proportion} - \eqref{eq: ODE 1 restated again}. This point is written as $(z^* , u_1^* , u_2^*)$.

\subsubsection{Euler-Lagrange Equations}

It seems easiest to work with the contracted rate function in \eqref{eq: contracted rate function}. We first find a much simpler form for the contracted rate function.

\begin{lemma}
\begin{align} \label{eq: tilde L first}
 \hat{L}\big( \dot{x}, x , u \big) = \ell\bigg( \frac{\dot{z}_1}{\lambda_1(x,u)} \bigg)\lambda_1(x,u) +  \ell\bigg( \frac{\dot{x} + \dot{z}_1}{\lambda_2(x,u)} \bigg)\lambda_2(x,u) 
\end{align}
where $\dot{z}_1$ is such that
\begin{align} \label{eq: quadratic for dot z 1}
\dot{z}_1 \big( \dot{x} + \dot{z}_1 \big) = \lambda_1(x,u) \lambda_2(x,u).
\end{align}
One must make sure to choose the root of the above quadratic that is such that $\dot{z}_1 \geq 0$. 

\end{lemma}
\begin{proof}
It must be that 
\[
\dot{x} = \dot{z}_2 - \dot{z}_1.
\]
Thus, upon substitution, the Lagrangian assumes the form in \eqref{eq: tilde L first}. We must choose $\dot{z}_1$ to minimize \eqref{eq: tilde L first}, subject to the constraint $0\leq \dot{z}_1 \leq \dot{x}$. Differentiating \eqref{eq: tilde L first} with respect to $\dot{z}_1$, and setting the derivative to zero, we obtain that
\begin{align}
 \log\bigg( \frac{\dot{z}_1}{\lambda_1(x,u)} \bigg)  +  \log\bigg( \frac{  \dot{z}_1 + \dot{x}}{\lambda_2(x,u)} \bigg) = 0.
\end{align}
Exponentiating both sides, we obtain \eqref{eq: quadratic for dot z 1}.
\end{proof}
Adapted to the calcium-signalling model of this section, the original Euler-Lagrange equations in \eqref{eq: reduced EL 1}-\eqref{eq: reduced EL 4} are thus, for $i\in \lbrace 1,2 \rbrace$,
 \begin{align}
\frac{\partial \hat{L}}{\partial x} + \sum_{k=1}^2 \eta^k  \frac{\partial A^k}{\partial x}  &= \frac{d}{dt} \frac{\partial \hat{L}}{\partial \dot{x}} \label{eq: EL simplified} \\
  \frac{\partial \hat{L}}{\partial u}+ \frac{d\eta^i}{dt}  + \sum_{k=1}^2 \eta^k \frac{\partial A^k}{\partial u} &= 0 \\
   \frac{du}{dt} - A(u(t),x(t) )&= 0 \\
  \eta^i(T) &= 0.
 \end{align}
Upon making substitutions, \eqref{eq: EL simplified} leads to a second order ODE of the form $\frac{d^2 x}{dt^2} = \ldots$. Next, we compute that the first order derivatives are
\begin{lemma}
\begin{align}
\frac{\partial \hat{L}}{\partial \dot{x}} &= \log\bigg( \frac{\dot{z}_1}{\lambda_1(x,u)} \bigg) \frac{\partial \dot{z}_1}{\partial \dot{x}} + \log\bigg( \frac{\dot{x} + \dot{z}_1}{\lambda_2(x,u)} \bigg)\bigg(1+ \frac{\partial \dot{z}_1}{\partial \dot{x}} \bigg) \text{ where }\\
 \frac{\partial \dot{z}_1}{\partial \dot{x}} &=- \frac{\dot{z}_1}{2\dot{z}_1 + \dot{x}} . \label{eq: dot z1 x}
\end{align}
Notice also that the derivative assumes the form
\begin{align}
\frac{d}{dt} \frac{\partial \hat{L}}{\partial \dot{x}} =\frac{\partial^2 \hat{L}}{\partial \dot{x}^2} \ddot{x} +  \frac{\partial^2 \hat{L}}{\partial \dot{x} \partial x} \dot{x}+\sum_{j=1}^2  \frac{\partial^2 \hat{L}}{\partial \dot{x} \partial u_j}  \frac{du_j}{dt}
\end{align}
\end{lemma}
In the next lemma we collect some more identities. The proof is neglected (one just applies Calculus to compute the derivatives).
\begin{lemma}
The second derivatives are
\begin{align}
 \frac{\partial^2 \dot{z}_1}{\partial \dot{x}^2} &= \frac{\dot{z}_1}{(2\dot{z}_1 + \dot{x})^2} \bigg( 1 + \frac{\dot{x}}{2\dot{z}_1 + \dot{x}} \bigg) \\
 \frac{\partial^2 \hat{L}}{\partial \dot{x}^2} &= 2 \frac{\partial^2 \dot{z}_1}{\partial \dot{x}^2}\log\bigg( \frac{\dot{z}_1 (\dot{x} + \dot{z}_1)}{\lambda_1(x,u) \lambda_2(x,u)} \bigg) + \frac{\lambda_1(x,u)}{\dot{x} + \dot{z}_1} \bigg( 1 + \frac{\partial \dot{z}_1}{\partial \dot{x}} \bigg) \\
 \frac{\partial^2 \hat{L}}{\partial \dot{x} \partial x} &= - \frac{\partial \lambda_1(x,u)}{\partial x} \frac{1}{\lambda_1(x,u)}\frac{\partial \dot{z}_1}{\partial \dot{x}} - \frac{\partial \lambda_2(x,u)}{\partial x} \frac{1}{\lambda_2(x,u)}\bigg(1+ \frac{\partial \dot{z}_1}{\partial \dot{x}} \bigg) \\
  \frac{\partial^2 \hat{L}}{\partial  \dot{x} \partial u_i} &= - \frac{\partial \lambda_1(x,u)}{\partial u_i} \frac{1}{\lambda_1(x,u)}\frac{\partial \dot{z}_1}{\partial \dot{x}} - \frac{\partial \lambda_2(x,u)}{\partial u_i} \frac{1}{\lambda_2(x,u)}\bigg(1+ \frac{\partial \dot{z}_1}{\partial \dot{x}} \bigg)
\end{align}

\end{lemma}

\subsection{Optimal Trajectories}

We wish to predict the probability of the proportion of open calcium channels being significantly greater than its equilibrium value at some time $T$. We write the `target' proportion of channels to be $\hat{x} \in (x_*, 1]$. We thus wish to compute the trajectory $(x,u) \in   \mathcal{A} \subset   \mathcal{C}^2([0,T],\mathbb{R}) \times  \mathcal{C}^1([0,T],\mathbb{R}^2)$ such that
\begin{align}
\tilde{\mathcal{J}}_T(x,u) = \inf \big\lbrace \tilde{\mathcal{J}}_T(y,v) \; : \; (y,v) \in \mathcal{A} \big\rbrace
\end{align}
and  $\mathcal{A} \subset   \mathcal{C}^2([0,T],\mathbb{R}) \times  \mathcal{C}^1([0,T],\mathbb{R}^2)$ consists of all $(x,u)$ such that $x(0) = x^*(0)$, 
\begin{align}
x(T) = \hat{x}, \label{eq: x T constraint}
\end{align}
 $u(0) = u^*$ and for all $t\in [0,T]$,
\begin{align}
   \frac{du}{dt} - A(u(t),x(t) ).
\end{align}
The trajectory that minimizes $\tilde{\mathcal{J}}_T$ indicates the most likely means that a significant proportion of calcium channels open over the time interval $[0,T]$ due to stochastic effects. The second order ODE for $x(t)$ may be inferred from the previous section. The dynamics of the Lagrange multiplier $\eta(t)$ is given by, and the boundary value at time $T$ is
\begin{align}
\eta^1(T) &= 0 \label{eq: eta constraint 1} \\
\eta^2(T) &=0. \label{eq: eta constraint 2}
\end{align}
Our system thus amounts to a boundary value ODE. It is solved using the shooting method. The unknown derivatives at time $0$ (which must be solved for using the shooting method) are $\dot{\eta}^1(0)$ , $\dot{\eta}^2(0)$, $\dot{x}(0)$. There are three constraints at time $T$ that are used to solve for the unknowns at time $0$:  \eqref{eq: x T constraint},  \eqref{eq: eta constraint 1} and \eqref{eq: eta constraint 2}.

\begin{figure}
\includegraphics[width=15cm]{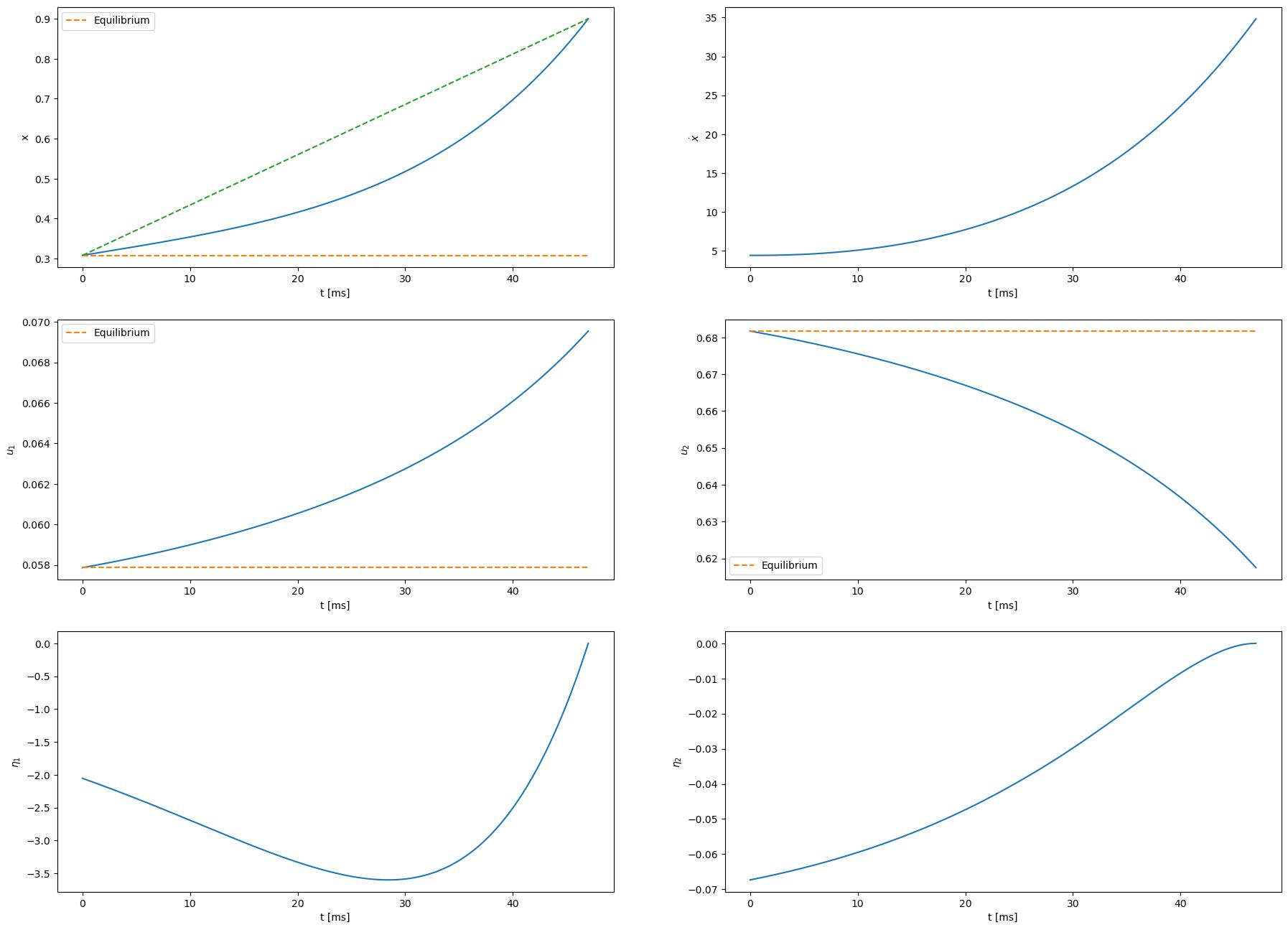}
\end{figure}

\subsection{Discussion}
We have demonstrated that Large Deviation Theory can be used to compute the most likely trajectory followed by $N$ stochastic channels / diffusing calcium in making the spark to wave transition \cite{Dupont2016}. This model is relatively simple, particularly insofar as the calcium concentration is spatially homogeneous. In a subsequent work, we will employ a spatially-distributed model of the calcium to compute the optimal trajectories.

\section{Proofs}
To prove Theorem \ref{Theorem Main LDP}, we apply the Contraction Principle \cite{Dembo1998} to the Large Deviation Principle for the independent system.
In other words, we are going to perform a change of variable, which will transfer known results for the Large Deviations of independent Poisson Processes \cite[Chapter 6]{Weiss1995}. We first recount the Large Deviations for independent Poisson Processes (recalling the definition of $\lbrace y_{\alpha}(t) \rbrace_{\alpha\in\mathscr{R}}$ in \eqref{eq: Z definition next}).


\begin{theorem}\label{Theorem Sanov}
The sequence of probability laws of $ (y_{\alpha})_{\alpha\in \mathscr{R}} \in \tilde{\mathcal{D}}([0,\infty),\mathbb{R}^+)^M$ satisfy a Large Deviation Principle with good rate function.
 That is, for $\mathcal{O},\mathcal{A} \subset  \tilde{\mathcal{D}}([0,\infty),\mathbb{R}^+)^M$, with $\mathcal{O}$ open and $\mathcal{A}$ closed,
\begin{align}
\lsup{N}N^{-1}\log \mathbb{P}\big(  y \in \mathcal{A} \big) &\leq - \inf_{\alpha \in \mathcal{A}}\mathcal{I}(\alpha) \\
\linf{N}N^{-1}\log \mathbb{P}\big( y \in \mathcal{O} \big) &\geq - \inf_{\alpha \in \mathcal{O}}\mathcal{I}(\alpha) ,
\end{align}
and the rate function is
\begin{align}
\mathcal{I}(y) &= \left\lbrace \begin{array}{c}
\infty \text{ in the case that }y_{\alpha} \text{ is not absolutely continuous for some }\alpha \in \mathscr{R} \\
\sum_{\alpha \in \mathscr{R}} \int_0^\infty \ell\big( \dot{y}_{\alpha}(r) \big)dr \text{ otherwise,  }
\end{array}\right.   \\
\ell(r) &= r\log r -r +1.
\end{align}
 Furthermore the level sets of $\mathcal{I}$ are compact.
\end{theorem}
\begin{proof}

The Large Deviations Principle for arbitrarily long finite time intervals is proved in \cite[Chapter 5]{Weiss1995}, with corresponding rate function
\begin{equation}
\mathcal{I}(y) = \sum_{\alpha \in \mathscr{R}} \int_0^{\infty} \ell\big( \dot{y}_{\alpha}(r) \big)dr.
\end{equation}
\end{proof}

The first result that we must prove is that the Large Deviations Principle must hold for sets with reaction rates bounded away from zero.
\begin{lemma}\label{Lemma Main LDP}
Suppose that $\mathcal{O},\mathcal{A} \subseteq \Upsilon$, with $\mathcal{O}$ open and $\mathcal{A}$ closed, are such that for any $T > 0$,
\begin{align}
\inf_{ (z,x,u) \in \mathcal{A} \cup \mathcal{O}} \inf_{\alpha\in\mathscr{R}} \inf_{s\in [0,T]} \lambda_\alpha(x(s),u(s)) > 0. \label{eq: set infimum lambda}
\end{align}
Then the sequence of probability laws of $\big( z,x,u \big)$ satisfy a Large Deviation Principle on the space $\Upsilon$, i.e.
\begin{align}
\lsup{N}N^{-1}\log \mathbb{P}\big( (z,x,u) \in \mathcal{A} \big) &\leq - \inf_{\beta \in \mathcal{A}}\mathcal{J}(\beta) \\
\linf{N}N^{-1}\log \mathbb{P}\big( (z,x,u) \in \mathcal{O} \big) &\geq - \inf_{\beta\in \mathcal{O}}\mathcal{J}(\beta).
\end{align}
$\mathcal{J}$ is lower-semi-continuous and the level sets of $\mathcal{J}$ are compact.
\end{lemma}
\begin{proof}
We are going to prove the Large Deviation Principle using a contraction principle. To this end, define a mapping $\Psi: \Upsilon \to \tilde{\mathcal{D}}([0,\infty),\mathbb{R}^+)^M$ as follows. For any $(z,x,u) \in \Upsilon$, define 
\begin{align}
\Lambda_{\alpha}(t) = \int_0^t \lambda_{\alpha}(x(s),u(s)) ds.
\end{align}
Let $\Lambda^{-1}_{\alpha} \in  \tilde{\mathcal{D}}([0,\infty),\mathbb{R}^+$ be such that
\begin{align}
\Lambda^{-1}_{\alpha}(t) = \inf\big\lbrace s\geq 0 \; : \Lambda_{\alpha}(s) = t \big\rbrace . 
\end{align}
Note that, since $\inf_{s\in [0,t]}\lambda_{\alpha}(x(s),u(s)) > 0$, it must be that $\Lambda^{-1}_{\alpha}$ is the function inverse of $\Lambda_{\alpha}$. Write
\begin{align}
\tau_{\alpha} = \lim_{t\to\infty} \Lambda_{\alpha}(t) ,
\end{align}
and note that $\tau_{\alpha}$ could be $\infty$.
We define $\Psi(z,x,u) = \big( w_{\alpha} \big)_{\alpha\in \mathscr{R}}$, where for any $t < \tau_{\alpha}$,
\begin{align}
w_{\alpha}(t) = z_{\alpha}\big( \Lambda^{-1}_{\alpha}(t) \big).
\end{align}
In the case that $\tau_{\alpha} < \infty$, for all $t\geq \tau_{\alpha}$, we define $w_{\alpha}(t) = \lim_{s\to \tau_{\alpha}^-}w_{\alpha}(s)$.
One observes that over sets of the form \eqref{eq: set infimum lambda}, $\Psi$ is continuous.  Furthermore, $w_{\alpha}(t) = y_{\alpha}(t)$. 
We wish to apply the Inverse Contraction Principle \cite{Dembo1998}, (noting also Lemmas \ref{Theorem Sanov} and \ref{Lemma Exponential Tightness}). In order that this theorem applies, we must also prove that $\Psi$ is one-to-one. This is proved in Lemma \ref{One to One Lemma}.


\end{proof}
\begin{lemma} \label{One to One Lemma}
Over any set $\mathcal{A}_{\epsilon}$ of the form, for $\epsilon > 0$,
\begin{align}
\mathcal{A}_\epsilon = \inf_{ (z,x,u) \in \Upsilon} \inf_{\alpha\in\mathscr{R}} \inf_{s\in [0,T]} \lambda_\alpha(x(s),u(s)) \geq \epsilon. \label{eq: set infimum lambda restated}
\end{align}
 $\Psi$ is one-to-one.
\end{lemma}
\begin{proof}
Let $(z,x,u) , (\hat{z},\hat{x},\hat{u}) \in \Upsilon$. Write
\begin{align}
\Lambda_{\alpha}(t) &= \int_0^t \lambda_{\alpha}(x(s),u(s)) ds \\
\hat{\Lambda}_{\alpha}(t) &= \int_0^t \lambda_{\alpha}( \hat{x}(s), \hat{u}(s)) ds \\
w_{\alpha}(t) &= z_{\alpha}\big( \Lambda^{-1}_{\alpha}(t) \big) \\
\hat{w}_{\alpha}(t) &= \hat{z}_{\alpha}\big( \hat{\Lambda}^{-1}_{\alpha}(t) \big).
\end{align}
It is noted in Lemma that there is a constant $C_t > 0$ such that
\begin{align}
\sup_{\alpha\in\mathscr{R}} \Lambda_{\alpha}(t)\leq C_t
\end{align}
We are going to demonstrate that for any $t \geq 0$, there exists a constant $c_{\epsilon,t}$ such that for all $s\leq t$,
\begin{align}
\tilde{d}^{\circ}_s( z_{\alpha}, \hat{z}_{\alpha}) \leq  \tilde{d}^{\circ}_{C_s}( w_{\alpha}, \hat{w}_{\alpha}) + c_{\epsilon,t} \tilde{d}^{\circ}_s( z_{\alpha}, \hat{z}_{\alpha}). \label{eq: to show gronwall one to one}
\end{align}
An application of Gronwall's Inequality to \eqref{eq: to show gronwall one to one} implies that if $w_{\alpha} = \hat{w}_{\alpha}$ for all $\alpha \in \mathscr{R}$, then necessarily  $z_{\alpha} = \hat{z}_{\alpha}$. This implies the Lemma.

With an aim of proving \eqref{eq: to show gronwall one to one}, it follows from Calculus that
\begin{align}
\Lambda^{-1}_{\alpha}(t) &= \int_0^t \lambda^{-1}_{\alpha}(x(s),u(s)) ds \\
\hat{\Lambda}^{-1}_{\alpha}(t) &= \int_0^t \lambda^{-1}_{\alpha}( \hat{x}(s), \hat{u}(s)) ds .
\end{align}
Our assumptions on the set $\mathcal{A}_{\epsilon}$ imply that
\begin{align}
\Lambda^{-1}_{\alpha}(t) &\leq t \epsilon^{-1} \label{eq: reaction rate bound 1} \\
\hat{\Lambda}^{-1}_{\alpha}(t) &\leq t \epsilon^{-1} .\label{eq: reaction rate bound 2}
\end{align}
Define
\begin{align}
\tilde{z}_{\alpha}(t) = w_{\alpha} \big( \hat{\Lambda}(t)  \big) .
\end{align}
The triangle inequality implies that
\begin{align}
\tilde{d}^{\circ}_{t}(z_\alpha , \hat{z}_{\alpha}) \leq \tilde{d}^{\circ}_{t}(z_\alpha , \tilde{z}_{\alpha}) + \tilde{d}^{\circ}_{t}(\tilde{z}_\alpha , \hat{z}_{\alpha}).
\end{align}
The definition of the Skorohod metric implies that
\begin{align}
 \tilde{d}^{\circ}_{t}(z_\alpha , \tilde{z}_{\alpha}) \leq \sup_{s\leq t} \bigg| \int_0^s  \lambda_{\alpha}(x(r),u(r)) /  \lambda_{\alpha}(\hat{x}(r),\hat{u}(r)) dr - s \bigg|. \label{eq: skorohod 0}
 \end{align}
 Our assumption that the reaction rates and functions are Lipschitz, together with Lemma \ref{Lemma Topology on Upsilon} implies that there is a constant $\bar{C}_{\epsilon,t} > 0$ such that 
\begin{align}
\sup_{s\in [0,t]} \sup_{\alpha\in \mathscr{R}} \big| \lambda_{\alpha}(x(s),u(s))- \lambda_{\alpha}(\hat{x}(s),\hat{u}(s)) \big| \leq &\bar{C}_{\epsilon,t} \sup_{s\in [0,t]} \sup_{\alpha\in \mathscr{R}} \big| z_{\alpha}(s) -\hat{z}_{\alpha}(s) \big| \\ 
\leq &  \bar{C}_{\epsilon,t}  \sup_{\alpha\in \mathscr{R}}  \tilde{d}^{\circ}_t(z_{\alpha} , \hat{z}_{\alpha}). \label{eq: lips 1}
\end{align}
Since the reaction rates are (by assumption in the statement of the Lemma), bounded from below, we thus obtain that there is a constant $\grave{C}_{\epsilon,t}$ such that for all $s \leq \bar{t}$,
\begin{align}
\sup_{\alpha\in \mathscr{R}} \tilde{d}^{\circ}_{s}(z_\alpha , \tilde{z}_{\alpha}) \leq \grave{C}_{\epsilon,\bar{t}}  \sup_{\alpha\in \mathscr{R}}  \tilde{d}^{\circ}_s(z_{\alpha} , \hat{z}_{\alpha}).
 \end{align}
 The definition of the Skorohod Metric, together with the bounds in \eqref{eq: reaction rate bound 1} also implies that 
 \begin{align}
  \tilde{d}^{\circ}_{t}(\tilde{z}_\alpha , \hat{z}_{\alpha}) \leq  \tilde{d}^{\circ}_{C_t}(w_\alpha , \hat{w}_{\alpha}).
 \end{align}
 We have thus proved \eqref{eq: to show gronwall one to one}.
\end{proof}

The following exponential tightness result is a standard requirement for Large Deviation Principles.
\begin{lemma} \label{Lemma Exponential Tightness}
For any $L > 0$, there exists a compact set $\mathcal{K}_{L} \subset \Upsilon$ such that
\begin{align}
\lsup{N} N^{-1}\log \mathbb{P}\big( (z,x,u) \notin \mathcal{K}_L \big) \leq - L.
\end{align}
\end{lemma}
\begin{proof}
We can equivalently formulate the system in terms of an empirical measure, and the exponential tightness is an immediate consequence. In more detail, write 
\begin{align}
\hat{\mu}^N = N^{-1}\sum_{j\in I_N} \delta_{w^j} \in \mathcal{P}\big( \mathcal{D}\big( [0,T], \mathbb{Z}^+ \big)^{|\mathscr{R}|} \big). 
\end{align}
Here $w^j := \big(w^{j}_{\alpha} \big)_{ \alpha \in \mathscr{R}} \subset \mathcal{D}\big( [0,T], \mathbb{Z}^+ \big)^{|\mathscr{R}|}$ are inhomogeneous counting processes, i.e. they are such that
\begin{align}
\mathbb{P}\big( w^{j}_{\alpha}(t+\Delta) =w^{j}_{\alpha}(t) +1 \; | \; \mathcal{F}_t \big) &\simeq \Delta \lambda_{\alpha}(x(t),u(t) )+ O(\Delta^2) \\
\mathbb{P}\big( w^{j}_{\alpha}(t+\Delta) = w^{j}_{\alpha}(t)   \; | \; \mathcal{F}_t \big) &\simeq 1 - \Delta \lambda_{\alpha}(x(t),u(t) )+ O(\Delta^2) .
\end{align}
Here
\begin{align}
x(t) =& x(0) + N^{-1} \sum_{j\in I_N} \sum_{\alpha\in\mathscr{R}}\xi_{\alpha}  w^j_{\alpha}(t) \\
\frac{du}{dt} =& A(u(t),x(t))
\end{align}
We then find that, substituting $z_{\alpha}(t) = N^{-1}\sum_{j=1}^N w^j_{\alpha}(t)$, the above system has the same probability law as the original system. Write $\mathcal{D}_*\big( [0,T], \mathbb{Z}^+ \big)^{|\mathscr{R}|} \big) \subseteq \mathcal{D}\big( [0,T], \mathbb{Z}^+ \big)^{|\mathscr{R}|}$ to consist of all processes that are (i) equal to $0$ at time $0$, (ii) non-decreasing.

More precisely, we see that there exists a continuous mapping $\Psi: \mathcal{P}\big( \mathcal{D}_*\big( [0,T], \mathbb{Z}^+ \big)^{|\mathscr{R}|} \big)$ such that $(z,x,u) = \Psi\big(\hat{\mu}^N\big)$, with unit probability. For a positive number $L>0$, we are going to define a compact set $\mathcal{K}_L \subseteq  \mathcal{P}\big( \mathcal{D}_*\big( [0,T], \mathbb{Z}^+ \big)^{|\mathscr{R}|} \big)$ such that
\begin{align}
\lsup{N} N^{-1}\log \mathbb{P}\big( \hat{\mu}^N \notin \mathcal{K}_L \big) \leq - L.
\end{align}
This suffices for the lemma because the continuity of $\Psi$ then implies that $\Psi(\mathcal{K}_L)$ is compact.

For a positive integer $p\geq 1$, write $\mathcal{U}_{p,T} \subseteq \mathcal{D}_*\big( [0,T], \mathbb{Z}^+ \big)^{|\mathscr{R}|}$ to consist of all paths that are less than or equal to $p$ at time $T$. Its easy to check that $\mathcal{U}_p$ is compact with respect to the Skorohod Topology (there are at most $p$ `spike times', and these spike times must be in the compact time interval $[0,T]$). Write $Y^{p,j}_\alpha$ to be independent counting processes of unit intensity. Thus the time-rescaled representation of Poisson Processes \cite{Anderson2015} means that we can write
\begin{align}
w^j_{\alpha}(t) = Y^j_\alpha\bigg( \int_0^t \lambda_{\alpha}(x(s),u(s)) ds \bigg)
\end{align}
For another number $b_p > 0$, we find that
\begin{align}
\mathbb{P}\big( \hat{\mu}^N(\mathbf{w})(\mathcal{U}_{p,T}) < 1 - b_p \big) \leq \mathbb{P}\big( \hat{\mu}^N(\mathbf{Y}) \big(\mathcal{U}_{p,KT}\big) < 1 - b_p \big) .
\end{align}
since $\lambda \leq K$ uniformly. Thanks to Chernoff's Inequality, for a constant $c > 0$, 
\begin{align}
 \mathbb{P}\big( \hat{\mu}^N(\mathbf{Y}) \big(\mathcal{U}_{p,KT}\big) < 1 - b_p \big) \leq & \mathbb{P}\bigg( \sup_{\alpha \in \mathscr{R}} N^{-1}\sum_{j=1}^N \chi\big\lbrace w^j_{\alpha}(KT) > p \big\rbrace \geq b_p \bigg) \\
\leq & \big| \mathscr{R} \big| \mathbb{E}\bigg[ \exp\bigg( c \sum_{j=1}^N\chi\big\lbrace w^j_{\alpha}(KT) > p \big\rbrace  - Nc b_p \bigg) \bigg] \\
=& \big|\mathscr{R}\big| \bigg\lbrace1 + \mathbb{P}\big(w^j_{\alpha}(KT) > p \big)  \big(\exp(c) - 1 \big) \bigg\rbrace^N \exp\big( - Nc b_p \big)  
\end{align}
Through taking $p$ large enough, and $1\ll c \ll  -\log\mathbb{P}\big(w^j_{\alpha}(KT) > p \big) $, we find that
\begin{align}
  \mathbb{P}\big( \hat{\mu}^N(\mathbf{w})(\mathcal{U}_{p,T}) < 1 - b_p \big) \leq \exp\big( -p N \big).
\end{align}
Now define $\mathcal{K}_L$ to consist of all measures $\mu$ such that for all $p \geq p_L$ (for an integer $p_L$ to be specified below),
\begin{align}
\mu( \mathcal{U}_{p,KT}) \geq b_p.
\end{align}
We thus find through a union of events bound that
\begin{align}
\mathbb{P}\bigg( \hat{\mu}^N \notin \mathcal{K}_L \bigg) \leq  &\sum_{p=p_L}^{\infty}\mathbb{P}\bigg(  \hat{\mu}^N(\mathbf{Y}) \big(\mathcal{U}_{p,KT}\big) < 1 - b_p  \bigg)  \\
\leq & \sum_{p=p_L}^{\infty}\exp\big(-pN \big).
\end{align}
Thus for large enough $p_L$, it must be that for all $N\geq 1$,
\begin{align}
N^{-1} \log \mathbb{P}\bigg( \hat{\mu}^N \notin \mathcal{K}_L \bigg) \leq - L.
\end{align}

\end{proof}
We finally note the proof of Lemma \ref{Lemma Topology on Upsilon}.
\begin{proof}
It is clear from the definition of $x$ that there is a universal constant such that
\begin{equation}
\norm{x(s) - \tilde{x}(s)} \leq C\sup_{\alpha\in\mathscr{R}}| z_{\alpha}(s) - \tilde{z}_{\alpha}(s) |.
\end{equation}
It is demonstrated in Billingsley \cite{Billingsley1999} that convergence in the Skorohod metric implies convergence in the supremum norm, i.e as $d^{\circ}_\infty(z,\tilde{z}) \to 0$, it must be that
\begin{equation}
\sup_{\alpha\in\mathscr{R}}\sup_{0\leq s \leq t}\big| z_{\alpha}(s) - \tilde{z}_{\alpha}(s) \big| \to 0.
\end{equation} 
The Lemma now follows from a standard application of Gronwall's Inequality.
\end{proof}

\bibliographystyle{plain}
\bibliography{CRN}

\end{document}